\documentclass{colt2012} 

\usepackage{algorithm,algorithmic}
\usepackage{amssymb, multirow, paralist, color}
\newtheorem{thm}{Theorem}

\usepackage{enumitem}

\def \w {\mathbf{w}}
\def \v {\mathbf{v}}

\def \x {\mathbf{x}}

\def \x {\mathbf{x}}

\def \z {\mathbf{z}}

\def \y {\mathbf{y}}
\def \u {\mathbf{u}}

\def \y {\mathbf{y}}

\def \x {\mathbf{x}}

\def \z {\mathbf{z}}
\def \u {\mathbf{u}}

\def \w {\mathbf{w}}

\def \v {\mathbf{v}}

\begin{document}

\title[ ]{Frank-Wolfe Method is Automatically Adaptive to Error Bound Condition}
\author{\Name{Yi Xu} \Email{yi-xu@uiowa.edu}\\
\Name{Tianbao Yang} \Email{tianbao-yang@uiowa.edu}\\
\addr Department of Computer Science, The University of Iowa, Iowa City, IA 52242
}

\maketitle
\vspace*{-0.5in}
\begin{center}{October 10, 2018}\end{center}

\begin{abstract}
Error bound condition has recently gained revived interest  in optimization. It has been leveraged to derive faster convergence for many popular algorithms, including subgradient methods, proximal gradient method and accelerated proximal gradient method. However, it is still unclear whether the Frank-Wolfe (FW) method can enjoy faster convergence under error bound condition. In this short note, we give an affirmative answer to this question. We show that the FW method (with a line search for the step size) for optimization over a strongly convex set is automatically adaptive to the error bound condition of the problem. In particular, the iteration complexity  of FW can be characterized by  $O(\max(1/\epsilon^{1-\theta}, \log(1/\epsilon)))$ where $\theta\in[0,1]$ is a constant that characterizes the error bound condition. Our results imply that if the constrained set is characterized by a strongly convex function and the objective function can achieve a smaller value outside the considered domain, then the FW method enjoys a fast rate of $O(1/t^2)$. 
\end{abstract}

\section{Introduction}
In this draft, we consider the following constrained convex optimization problem:
\begin{align}\label{eqn:opt}
\min_{\x\in\Omega} f(\x)
\end{align}
where $f(\w)$ is a smooth function and $\Omega\subseteq\mathbf E$ is a bounded strongly convex set. We assume that linear optimization over $\Omega$ is much more cheaper than projection onto $\Omega$, which makes the FW method more suitable for solving the above problem than gradient methods. The goal of this paper is to show that the FW method is automatically adaptive to an error bound condition of the optimization problem. Below, we will first review the FW method and the error bound condition. In next section, we will prove that the FW method is automatically adaptive to the error bound condition. 

The original FW method, introduced by \cite{frank1956algorithm} (a.k.a. Conditional Graident method~\citep{levitin1966constrained}), is a projection-free fist-order method for minimizing smooth convex objective functions over a convex set. In recent years, the FW method has gained an increasing interest in large-scale optimization and machine learning (e.g., \citep{pmlr-v37-garbera15, freund2016new, nesterov2018complexity, narasimhan2018learning}). 	
Many existing works have shown the convergence rate of the standard FW method is $O(1/t)$ even for strongly convex objectives~\citep{clarkson2008coresets,hazan2008sparse,jaggi2013revisiting}, and in general the rate could not be improved. Under different assumptions or for some special cases, a series of works tried to get faster rates of the FW method and its variants~\citep{levitin1966constrained,demyanov1970approximate, dunn1979rates, guelat1986some,beck2004conditional,garber2013playing, lan2013complexity, lacoste2013affine, pmlr-v37-garbera15,lacoste2015global, lan2016conditional}. For example, for  minimizing smooth and strongly convex objective functions over a strongly convex set, \cite{pmlr-v37-garbera15} showed that the FM method enjoyed fast rate of $O(1/t^2)$.  

In this paper, we first consider the FW method shown in Algorithm~\ref{alg:0}, where $L_f$ denotes a smoothness constant of $f(\x)$ with respect to $\|\cdot\|$ such that $f(\x)\leq f(\y) + \nabla f(\y)^{\top}(\x - \y) + \frac{L_f}{2}\|\x - \y\|^2$ holds for any $\x, \y\in\Omega$. Note that both options for selecting the step size have been considered in the literature~\citep{jaggi2013revisiting, pmlr-v37-garbera15}. Option I requires evaluating the objective function but does not need to  know  the smoothness constant. Option II could be cheaper but requires knowing the Lipschitz constant of the gradient. Our analysis applies to both options. In the sequel, we will focus on option I, with which we have
\begin{align}\label{eqn:key}
f(\x_{t+1})&\leq f(\x_t + \eta(\y_t - \x_t)), \forall \eta\in[0,1]\nonumber\\
&\leq f(\x_t) + \eta(\y_{t} - \x_t)^{\top}\nabla f(\x_t) + \frac{\eta^2L_f}{2}\|\y_t - \x_t\|^2, \forall \eta\in[0,1]
\end{align}	
Note that for option II, the second inequality above still holds.
	
We consider the following definition of error bound condition for the optimization problem~\eqref{eqn:opt}. 
\begin{definition}[H\"{o}lderian error bound (HEB)]\label{def:1}
A function $f(\x)$ is said to satisfy a HEB condition on  $\Omega$ if there exist $\theta\in[0,1]$ and $0<c<\infty$ such that for any $\x\in\Omega$
\begin{align}\label{eqn:leb}
\min_{\w\in\Omega_*}\|\x - \w\|\leq c(f(\x) - f_*)^{\theta}.
\end{align}
where $\Omega_*$ denotes the optimal set of $\min_{\x\in\Omega}f(\x)$ and $f_*$ denotes the optimal objective value.  
\end{definition}
It is notable that $\theta=0$ is a trivial condition since it always hold due to that  $\Omega$ is a compact set.  The above HEB condition has been considered for deriving faster convergence of subgradient methods~\citep{yang2018rsg}, proximal gradient method~\citep{liu2017adaptive}, accelerated gradient method~\citep{xu2016homotopy}, and stochastic subgradient methods~\citep{pmlr-v70-xu17a}. It has been shown that many problems satisfy the above condition~\citep{xu2016homotopy, pmlr-v70-xu17a, DBLP:conf/nips/XuLLY17, liu2017adaptive, yang2018rsg}. For example, when functions are semi-algebraic and ``regular" (for instance, continuous), the above inequality is known to hold on any compact set (c.f. \citep{bolte2017error} and references therein). 
		
The last definition in this section is regarding the strongly convex set. 
\begin{definition}
A convex set $\Omega$ is a $\alpha$-strongly convex with respect to $\|\cdot\|$ if for any $\x, \y\in\Omega$, any $\gamma\in[0,1]$ and any vector $\z\in\mathbf E$ such that $\|\z\|=1$, it holds that 
\begin{align*}
\gamma\x + (1-\gamma)\y + \gamma(1-\gamma) \frac{\alpha}{2}\|\x - \y\|^2\z\in\Omega. 
\end{align*}
\end{definition}
{\bf Remark. }Many previous works~(e.g., \citep{levitin1966constrained,demyanov1970approximate, dunn1979rates, pmlr-v37-garbera15}) considered this condition of feasible set when studying the FW method.
\begin{algorithm}[t]
\caption{Frank-Wolfe Method} \label{alg:0}
\begin{algorithmic}
\STATE Initilization: $\x_0 \in\Omega$
\FOR{$t=0,\ldots, T$}
    	\STATE Compute $\y_t \in \arg\min_{\y\in\Omega}\nabla f(\x_t)^{\top}\y$\;
	\STATE Option I: Set $\eta_t = \arg\min_{\eta\in[0,1]}f(\x_t + \eta(\y_t - \x_t))$ \;
	\STATE Option II: Set $\eta_t = \arg\min_{\eta\in[0,1]}\eta(\y_t - \x_t)^{\top}\nabla f(\x_t) + \frac{\eta^2L_f}{2}\|\y_t - \x_t\|^2$ \;
 	\STATE Compute $\x_{t+1} =\x_t + \eta_t (\y_t - \x_t) $\;
\ENDFOR
\end{algorithmic}
\end{algorithm}

\section{Adaptive Convergence of the FW method}
In this section, we show that the FW method is {\it automatically} adaptive to the HEB condition, enjoying a faster convergence rate than the standard $O(1/t)$ rate without the knowledge of the HEB condition. 

We first prove the following lemma. 
\begin{lemma}\label{lem:1}
Assume $f(\x)$ obeys the HEB condition on $\Omega$ with $\theta\in[0,1]$, then it holds that 
\begin{align*}
\|\nabla f(\x)\|_*\geq \frac{1}{c}(f(\x) - f_*)^{1-\theta}.
\end{align*}
\end{lemma}
\begin{proof}
Let $\x_*$ denote the optimal solution in $\Omega_*$ that is closest to $\x$ measured in $\|\cdot\|$.  
By convexity of $f(\cdot)$, we have
\begin{align*}
f(\x_*)\geq f(\x) + \nabla f(\x)^{\top}(\x_* - \x).
\end{align*}
Thus, 
\begin{align*}
f(\x) - f(\x_*)\leq \|\nabla f(\x)\|_* \|\x - \x_*\|\leq c(f(\x) - f_*)^{\theta} \|\nabla f(\x)\|_* .
\end{align*}
As a result, 
\begin{align*}
\|\nabla f(\x)\|_*\geq \frac{1}{c}(f(\x) - f_*)^{1-\theta}.
\end{align*}
\end{proof}
The second lemma is from \citep{pmlr-v37-garbera15}.
\begin{lemma}\label{lem:2}
For the FW method given in Algorithm~\ref{alg:0}, for $t=0,\ldots, $ we have
\begin{align*}
f(\x_{t+1}) - f_*\leq (f(\x_t) - f_*)\max\left\{\frac{1}{2}, \left(1 - \frac{\alpha \|\nabla f(\x_t)\|_*}{8L_f}\right)\right\}.
\end{align*}
\end{lemma}
Finally, we prove the following theorem. 
\begin{thm}
For every $t\geq 1$, we have
\begin{align*}
f(\x_t) -f _* \leq \left\{\begin{array}{lc}\frac{C}{(t+k)^{1/(1-\theta)}}& \text{ if }\theta\in [0,1)\\ \\  \rho^{t}(f(\x_0) - f_*)& \text{ otherwise}\end{array}\right.
\end{align*}
where $k\geq \max\left\{\frac{2 - 2^{1-\theta}}{2^{1-\theta} -1}, C'\right\}$,  $C\geq \max\left\{\frac{L_fD^2(1+k)^{\frac{1}{1-\theta}}}{2}, 2\left(\frac{C'}{M}\right)^{\frac{1}{1-\theta}}\right\}$, $C' = \frac{1}{1-\theta -\theta(2^{1-\theta}-1)}$, and $\rho=\max\left\{\frac{1}{2}, 1 - \frac{\alpha}{8cL_f}\right\}$. 
\end{thm}
{\bf Remark. }In order to find an $\epsilon$-approximate solution $\x_t$ such that $f(\x_t) - f_* \leq \epsilon$, the iteration complexity of FW method is $O(\max(1/\epsilon^{1-\theta}, \log(1/\epsilon)))$ with $\theta\in[0,1]$.
\begin{proof}
When $\theta=1$, the conclusion is trivial, which follows directly from Lemma~\ref{lem:2}. Next, we prove for $\theta\in[0,1)$. 
Let $\beta=1-\theta$. $h_t = f(\x_t) - f_*$. Combining Lemma~\ref{lem:1} and Lemma~\ref{lem:2}, we have
\begin{align}\label{eqn:max}
h_{t+1} \leq h_t \max\left\{\frac{1}{2}, 1 - \frac{\alpha}{8cL_f}h_t^{\beta}\right\} =  h_t \max\left\{\frac{1}{2}, 1 - M h_t^{\beta}\right\}
\end{align}

We prove by induction that $h_t\leq \frac{C}{(t+k)^{1/\beta}}$. 

For the case $t=1$, following~(\ref{eqn:key}) we have
\begin{align*}
h_1\leq  h_0(1-\eta) + \frac{L_f\eta^2D^2}{2} \leq \max\left\{ \frac{L_fD^2}{2}, h_0 \right\} \leq \frac{L_fD^2}{2}, \forall \eta\in[0,1],
\end{align*}
where we use the fact that $h_0 = f(\x_0) - f(\x_*) \le \nabla f(\x_*)^\top(\x_0 - \x_*) + \frac{L_f}{2}\|\x_0-\x_*\|^2 =  \frac{L_f}{2}\|\x_0-\x_*\|^2 \leq  \frac{L_fD^2}{2}$].
As long as $C/(1+k)^{1/\beta}\geq L_fD^2/2$, we have the conclusion holds for $t=1$.

Next, we consider $t\geq 1$. First assume that the max operation in~(\ref{eqn:max}) gives $1/2$, i.e., 
\begin{align*}
h_{t+1}\leq \frac{h_t }{2}\leq \frac{C}{2(t+k)^{1/\beta}}\leq \frac{C}{(t+1+k)^{1/\beta}}\frac{(t+1+k)^{1/\beta}}{2(t+k)^{1/\beta}}\leq \frac{C}{(t+1+k)^{1/\beta}},
\end{align*}
where the last inequality holds as long as 
\begin{align*}
\frac{(t+1+k)^{1/\beta}}{2(t+k)^{1/\beta}}\leq 1, \forall t\geq 1, \quad\text{i.e.,}\quad 1 + \frac{1}{t+k}\leq 2^\beta, \forall t\geq 1, \quad\text{i.e.,}\quad  k\geq \frac{2 - 2^{\beta}}{2^\beta -1}.
\end{align*}
Next, consider the case that the max operation is the second argument. In this case, if $h_t\leq\frac{C}{2(t+k)^{1/\beta}}$, the same conclusion holds under the above condition of $k$. Otherwise, $h_t\geq \frac{C}{2(t+k)^{1/\beta}}$. We have
\begin{align*}
h_{t+1}\leq h_t (1-Mh_t^{\beta})&\leq \frac{C}{(t+k)^{1/\beta}}\left(1 - M\left(\frac{C}{2}\right)^{\beta}\frac{1}{t + k} \right)\\
&\leq \frac{C}{(t+k+1)^{1/\beta}}\frac{(t+k+1)^{1/\beta}}{(t+k)^{1/\beta}}\left(1 - M\left(\frac{C}{2}\right)^{\beta}\frac{1}{t + k} \right)\\
&\leq  \frac{C}{(t+k+1)^{1/\beta}}\left(1 +\frac{C'}{t+k}\right)\left(1 - \frac{C'}{t+k}\right)
\end{align*}
To show the last inequality holds, we can set $C' = \frac{1}{\beta -(1-\beta)(2^{\beta}-1)} > 1$ and $C\geq 2(C'/M)^{1/\beta}$. To see this, we need to show that 
\begin{align*}
\log(1+ C'x) - \frac{1}{\beta}\log(1+x)\geq 0, \forall 0\leq x\leq 2^{\beta}-1. 
\end{align*}
In fact, due to $\frac{C'}{1+C'x} - \frac{1}{\beta(1+x)}\geq 0, \forall 0\leq x\leq 2^\beta -1$, 
it gives $1 + C'x\geq (1+x)^{1/\beta}$ holds for all $0\leq x\leq 2^\beta -1$. Plugging $x=1/(t+k)\leq 2^\beta -1$ into this inequality, we get what we want $1+C'/(t+k)\geq (1+\frac{1}{t+k})^{1/\beta}$.
\end{proof}

\section{Examples}
Lastly, we give examples exhibiting the HEB condition with $\theta=1/2$. In particular, let us 
 consider 
 \begin{align}\label{ex:1}
 \min_{g(\x)\leq r} f(\x)
 \end{align}
 where $g(\x)$ is a non-negative, strongly and smooth function. It is shown that $\Omega=\{\x: g(\x)\leq r\}$ is a strongly convex set~\citep{pmlr-v37-garbera15}. 
 \begin{lemma}
 Assume that $\min_{\x}f(\x)<\min_{g(\x)\leq r}f(\x)$ and there exists a $\x_0$ such that $g(\x_0) < r$, then the above problem satisfies HEB with $\theta=1/2$.
 \end{lemma}
 \begin{proof}
 We set $\Omega = \{\x: g(\x) \leq r \}$ and $\Omega_* = \arg\min_{g(\x)\leq r} f(\x)$, and we define an indicator function as follows,
 \begin{align*}
 I_{\Omega}(\x) = 
 \left\{ \begin{array}{rl}
         0 & \text{if~}\x\in\Omega, \\ 
        +\infty & \text{if~}\x\notin\Omega.
 \end{array}\right.
 \end{align*}
Then the problem of (\ref{ex:1}) can be written as
 \begin{align*}
 \min_{\x} \widehat f(\x) := f(\x) + I_{\Omega}(\x),
 \end{align*}
 and thus we also have $\Omega_* = \arg\min_{\x}\widehat f(\x)$.
 We only need to consider any fixed $\x_*\in\Omega_*$. By the condition of $g(\x_0) < r$ and Corollary 28.2.1 of~\citep{rockafellar1970convex}, there exists $\lambda^*\geq 0$ such that
 \begin{align}\label{ex1:ineq:1}
 \nonumber \widehat f(\x_*) = & \min_{\x} \widehat f(\x) = f(\x_*) = \min_{\x\in\Omega} f(\x) = \min_{\x} \{f(\x)+\lambda^*(g(\x)-r)\}\\
 \leq & f(\x_*)+\lambda^*(g(\x_*)-r) \leq f(\x_*),
 \end{align}
 where the first inequality is due to $\x_*\in\Omega_*$; the second inequality uses the fact that $\x_*\in\Omega_*$ hence $g(\x_*)-r\leq 0$.
 Then, equality holds for (\ref{ex1:ineq:1}), which implies $f(\x_*)+\lambda^*(g(\x_*)-r) = f(\x_*)$, that is,
  \begin{align}\label{ex1:ineq:2}
\lambda^*(g(\x_*)-r) = 0.
 \end{align}
 On the other hand, let $\u_* \in \arg\min_{\x} f(\x)$,  then based on the assumption of $\min_{\x}f(\x)<\min_{g(\x)\leq r}f(\x)$ we know $\u_* \notin \Omega_*$ hence $\u_* \notin \Omega$. By (\ref{ex1:ineq:1}), we also know
 \begin{align*}
 f(\u_*) < \min_{\x\in\Omega} f(\x) =  \min_{\x} \{f(\x)+\lambda^*(g(\x)-r)\} \leq f(\u_*)+\lambda^*(g(\u_*)-r),
 \end{align*}
 which implies 
  \begin{align}\label{ex1:ineq:3}
 \lambda^*(g(\u_*)-r)>0.
 \end{align}
 Since $\u_* \notin \Omega$, then $g(\u_*)-r > 0$. In order to have (\ref{ex1:ineq:3}), we need $\lambda^* > 0$. Thus, by (\ref{ex1:ineq:2}) we have
   \begin{align}\label{ex1:ineq:4}
g(\x_*)-r = 0.
 \end{align}
 For any such $\lambda^*>0$, then by Theorem 28.1 of~\citep{rockafellar1970convex}, we also have
 \begin{align}\label{ex1:ineq:5}
 \Omega_* = \{\x: g(\x) = r\} \cap \arg\min_{\x}\{f(\x) + \lambda^*(g(\x)-r)\}.
 \end{align}
 Since $g(\x)$ is strongly convex, $f(\x)$ is convex and $\lambda^*>0$, then $f(\x) + \lambda^*(g(\x)-r)$ is also strongly convex, implying that 
 $\v_* = \arg\min_{\x}\{f(\x) + \lambda^*(g(\x)-r)\}$ is a unique constant.
 Due to $\lambda^*>0$, $g(\v_*)$ is also a constant~\citep{li2017calculus}.
 By (\ref{ex1:ineq:5}) we have $g(\v_*) = g(\x_*) = r$.
 Therefore,
  \begin{align}\label{ex1:ineq:6}
 \Omega_* =  \arg\min_{\x}\{f(\x) + \lambda^*(g(\x)-r)\}.
 \end{align}
 By the strong convexity of $f(\x) + \lambda^*(g(\x)-r)$ we know for any $\x\in\Omega$ and $\x_*\in\Omega_*\subseteq \Omega$,
\begin{align*}
\frac{1}{c^2}\|\x-\x_*\|^2 \leq f(\x) + \lambda^*(g(\x)-r) - [f(\x_*) + \lambda^*(g(\x_*)-r)],
\end{align*}
where $c>0$.
Since $\lambda^*>0$, $g(\x)-r \leq 0$ and $g(\x_*)-r=0$, we get
 \begin{align*}
\frac{1}{c^2}\|\x-\x_*\|^2 \leq f(\x)  - f(\x_*).
\end{align*}
Therefore, for any $\x\in\Omega$
\begin{align*}
\min_{\w\in\Omega_*} \|\x - \w\| \leq c(f(\x) - f_*)^{1/2},
\end{align*}
which implies $\theta = 1/2$.
 \end{proof}

\bibliography{all}

\end{document}